\documentclass[journal]{IEEEtran}

\usepackage[sanserif,full]{complexity}
\usepackage{url}

\usepackage{graphicx}
\usepackage{amssymb}
\usepackage{amsmath}
\usepackage{amsthm}

\newtheorem{thm}{Theorem}
\newtheorem{defn}[thm]{Definition}

\newtheorem{prob}[thm]{Problem}

\begin{document}

\title{Certifying the Restricted Isometry Property is Hard}

\author{Afonso S.~Bandeira, Edgar Dobriban, Dustin G.~Mixon, William F.~Sawin
\thanks{A.S.~Bandeira is with the Program in Applied and Computational Mathematics, Princeton University, Princeton, New Jersey 08544 USA (e-mail: ajsb@math.princeton.edu).}
\thanks{E.~Dobriban is with the Department of Statistics, Stanford University, Stanford, California 94305 USA (e-mail: dobriban@stanford.edu).}
\thanks{D.G.~Mixon is with the Department of Mathematics and Statistics, Air Force Institute of Technology, Wright-Patterson AFB, Ohio 45433 USA (e-mail: dustin.mixon@afit.edu).}
\thanks{W.F.~Sawin is with the Department of Mathematics, Princeton University, Princeton, New Jersey 08544 USA (e-mail: wsawin@math.princeton.edu).}
\thanks{The authors thank Boris Alexeev for reading this manuscript and providing thoughtful comments and suggestions.
Bandeira was supported by NSF Grant No.~DMS-0914892, and Mixon was supported by the A.B.~Krongard Fellowship. The views expressed in this article are those of the authors and do not reflect the official policy or position of the United States Air Force, Department of Defense, or the U.S. Government.}%
}
\maketitle

\begin{abstract}
This paper is concerned with an important matrix condition in compressed sensing known as the restricted isometry property (RIP).
We demonstrate that testing whether a matrix satisfies RIP is $\NP$-hard.
As a consequence of our result, it is impossible to efficiently test for RIP provided $\P\neq\NP$.
\end{abstract}

\section{Introduction}

It is now well known that compressed sensing offers a method of taking few sensing measurements of high-dimensional sparse vectors, while at the same time enabling efficient and stable reconstruction~\cite{CandesT:it05}.
In this field, the restricted isometry property is arguably the most popular condition to impose on the sensing matrix in order to acquire state-of-the-art reconstruction guarantees:

\begin{defn}
\label{defn.rip}
We say a matrix $\Phi$ satisfies the $(K,\delta)$-\textit{restricted isometry property (RIP)} if
\begin{equation*}
(1-\delta)\|x\|^2
\leq\|\Phi x\|^2
\leq(1+\delta)\|x\|^2
\end{equation*}
for every vector $x$ with at most $K$ nonzero entries.
\end{defn}

To date, RIP-based reconstruction guarantees exist for Basis Pursuit~\cite{Candes:08}, CoSaMP~\cite{NeedellT:09} and Iterative Hard Thresholding~\cite{BlumensathD:09}, and the ubiquitous utility of RIP has made the construction of RIP matrices a subject of active research~\cite{BandeiraFMW:arxiv12}--\cite{Tao:07}.
Here, random matrices have found much more success than deterministic constructions~\cite{BandeiraFMW:arxiv12}, but this success is with high probability, meaning there is some (small) chance of failure in the construction.
Furthermore, RIP is a statement about the conditioning of all $\binom{N}{K}$ submatrices of an $M\times N$ sensing matrix, and so it seems computationally intractable to check whether a given instance of a random matrix fails to satisfy RIP; it is widely conjectured that certifying RIP for an arbitrary matrix is $\NP$-hard.
In the present paper, we prove this conjecture.

\begin{prob}
\label{prob.rip}
Given a matrix $\Phi$, a positive integer $K$, and some $\delta\in(0,1)$, does $\Phi$ satisfy the $(K,\delta)$-restricted isometry property?
\end{prob}

In short, we show that any efficient method of solving Problem~\ref{prob.rip} can be called in an algorithm that efficiently solves the $\NP$-complete subset sum problem.
As a consequence of our result, there is no method by which one can efficiently test for RIP provided $\P\neq\NP$.
This contrasts with previous work~\cite{KoiranZ:11}, in which the reported hardness results are based on less-established assumptions on the complexity of dense subgraph problems.

In the next section, we review the basic concepts we will use from computational complexity, and Section~3 contains our main result.

\section{A brief review of computational complexity}

In complexity theory, problems are categorized into complexity classes according to the amount of resources required to solve them.
For example, the complexity class $\P$ contains all problems which can be solved in polynomial time, while problems in $\EXP$ may require as much as exponential time.
Problems in $\NP$ have the defining quality that solutions can be verified in polynomial time given a certificate for the answer.
As an example, the graph isomorphism problem is in $\NP$ because, given an isomorphism between graphs (a certificate), one can verify that the isomorphism is legitimate in polynomial time.
Clearly, $\P\subseteq\NP$, since we can ignore the certificate and still solve the problem in polynomial time.

While problem categories provide one way to describe complexity, another important tool is the \textit{polynomial-time reduction}, which allows one to show that a given problem is ``more complex'' than another.
To be precise, a polynomial-time reduction from problem $A$ to problem $B$ is a polynomial-time algorithm that solves problem $A$ by exploiting an oracle which solves problem $B$; the reduction indicates that solving problem $A$ is no harder than solving problem $B$ (up to polynomial factors in time), and we say ``$A$ reduces to $B$,'' or $A\leq B$.
Such reductions lead to some of the most popular definitions in complexity theory:  
We say a problem $B$ is called $\NP$-\textit{hard} if every problem $A$ in $\NP$ reduces to $B$, and a problem is called $\NP$-\textit{complete} if it is both $\NP$-hard and in $\NP$.
In plain speak, $\NP$-hard problems are harder than every problem in $\NP$, while $\NP$-complete problems are the hardest of problems in $\NP$.

Contrary to popular intuition, $\NP$-hard problems are not merely problems that seem to require a lot of computation to solve.
Of course, $\NP$-hard problems have this quality, as an $\NP$-hard problem can be solved in polynomial time only if $\P=\NP$; this is an open problem, but it is widely believed that $\P\neq\NP$~\cite{Cook:online}.
However, there are other problems which seem hard but are not known to be $\NP$-hard (e.g., the graph isomorphism problem).
As such, while testing for RIP in the general case seems to be computationally intensive, it is not obvious whether the problem is actually $\NP$-hard.
Indeed, by the definition of $\NP$-hard, one must compare its complexity to the complexity of every problem in $\NP$.
To this end, notice that $A\leq B$ and $B\leq C$ together imply $A\leq C$, and so to demonstrate that a problem $C$ is $\NP$-hard, it suffices to show that $B\leq C$ for some $\NP$-hard problem $B$.

In the present paper, we demonstrate the hardness of certifying RIP by reducing from the following problem:

\begin{prob}
\label{prob.spark}
Given a matrix $\Psi$ and some positive integer $K$, do there exist $K$ columns of $\Psi$ which are linearly dependent?
\end{prob}

Problem~\ref{prob.spark} has a brief history in computational complexity.
First, McCormick~\cite{McCormick:phd83} demonstrated that the analogous problem of testing the girth of a transversal matroid is $\NP$-complete, and so by invoking the randomized matroid representation of Marx~\cite{Marx:tcs09}, Problem~\ref{prob.spark} is hard for $\NP$ under randomized reductions~\cite{AlexeevCM:arxiv11}.
Next, Khachiyan~\cite{Khachiyan:jc95} showed that the problem is $\NP$-hard by focusing on the case where $K$ equals the number of rows of $\Psi$; using a particular matrix construction with Vandermonde components, he reduced this instance of the problem to the subset sum problem.
Recently, Tillmann and Pfetsch~\cite{TillmannP:arxiv12} used ideas similar to McCormick's to strengthen Khachiyan's result: they prove Problem~\ref{prob.spark} is $\NP$-hard without focusing on such a specific instance of the problem.
Each of these complexity results use $M\times N$ matrices with integer entries whose binary representations take $\leq p(M,N)$ bits for some polynomial $p$; we will exploit this feature in our proof.


\section{Main result}

\begin{thm}
\label{thm.rip}
Problem~\ref{prob.rip} is $\NP$-hard.
\end{thm}

\begin{proof}
Reducing from Problem~\ref{prob.spark}, suppose we are given a matrix $\Psi$ with integer entries.
Letting $\mathrm{Spark}(\Psi)$ denote the size of the smallest collection of linearly dependent columns of $\Psi$, we wish to determine whether $\mathrm{Spark}(\Psi)\leq K$.
To this end, we take $P\leq2^{p(M,N)}$ to be the size of the largest entry in $\Psi$, and define $C=2^{\lceil \log_2 \sqrt{MN}P\rceil}$ and $\Phi=\frac{1}{C}\Psi$; note that we choose $C$ to be of this form instead of $\sqrt{MN}P$ to ensure that the entries of $\Phi$ can be expressed in $\poly(M,N)$ bits without truncation.
Of course, linear dependence between columns is not affected by scaling, and so testing $\Phi$ is equivalent to testing $\Psi$. 
In fact, since we plan to appeal to an RIP oracle, it is better to test $\Phi$ since the right-hand inequality of Definition~\ref{defn.rip} is already satisfied for every $\delta>0$:
\begin{align*}
\|\Phi\|_2
\leq\sqrt{MN}\|\Phi\|_{\mathrm{max}}
=\sqrt{MN}\frac{P}{C}
\leq1
\leq\sqrt{1+\delta}.
\end{align*}
We are now ready to state the remainder of our reduction:
For some value of $\delta$ (which we will determine later), ask the oracle if $\Phi$ is $(K,\delta)$-RIP; then
\begin{equation*}
\begin{array}{rcll}
\mbox{(i)} & \mbox{$\Phi$ is $(K,\delta)$-RIP} & \Longrightarrow & \mathrm{Spark}(\Psi)>K,\\
\mbox{(ii)} & \mbox{$\Phi$ is not $(K,\delta)$-RIP} & \Longrightarrow & \mathrm{Spark}(\Psi)\leq K.
\end{array}
\end{equation*}
The remainder of this proof will demonstrate (i) and (ii).

Note that (i) immediately holds for all choices of $\delta\in(0,1)$ by the contrapositive.
Indeed, $\mathrm{Spark}(\Psi)\leq K$ implies the existence of a nonzero vector $x$ in the nullspace of $\Phi$ with $\leq K$ nonzero entries, and $\|\Phi x\|^2=0<(1-\delta)\|x\|^2$ violates the left-hand inequality of Definition~\ref{defn.rip}.
For (ii), we also consider the contrapositive.
When $\mathrm{Spark}(\Psi)>K$, we have that every size-$K$ subcollection of $\Psi$'s columns is linearly independent.
Letting $\Psi_\mathcal{K}$ denote the submatrix of columns indexed by a size-$K$ subset $\mathcal{K}\subseteq\{1,\ldots,N\}$, this implies that $\lambda_\mathrm{min}(\Psi_\mathcal{K}^*\Psi_\mathcal{K}^{})>0$, and so $\det(\Psi_\mathcal{K}^*\Psi_\mathcal{K}^{})>0$.
Since the entries of $\Psi$ lie in $\{-P,\ldots,P\}$, we know the entries of $\Psi_\mathcal{K}^*\Psi_\mathcal{K}^{}$ lie in $\{-MP^2,\ldots,MP^2\}$, and since $\Psi_\mathcal{K}^*\Psi_\mathcal{K}^{}$ is integral with positive determinant, we must have $\det(\Psi_\mathcal{K}^*\Psi_\mathcal{K}^{})\geq1$.
In fact,
\begin{align*}
1
&\leq \det(\Psi_\mathcal{K}^*\Psi_\mathcal{K}^{})\\
&=\prod_{k=1}^K\lambda_k(\Psi_\mathcal{K}^*\Psi_\mathcal{K}^{})\\
&\leq \lambda_\mathrm{min}(\Psi_\mathcal{K}^*\Psi_\mathcal{K}^{})\cdot\lambda_\mathrm{max}(\Psi_\mathcal{K}^*\Psi_\mathcal{K}^{})^{K-1}\\
&\leq \lambda_\mathrm{min}(\Psi_\mathcal{K}^*\Psi_\mathcal{K}^{})\cdot\big(K\|\Psi_\mathcal{K}^*\Psi_\mathcal{K}^{}\|_{\mathrm{max}}\big)^{K-1},
\end{align*}
and so we can rearrange to get
\begin{align*}
\lambda_\mathrm{min}(\Phi_\mathcal{K}^*\Phi_\mathcal{K}^{})
&=\frac{1}{C^2}\lambda_\mathrm{min}(\Psi_\mathcal{K}^*\Psi_\mathcal{K}^{})\\
&\geq\frac{1}{C^2(KMP^2)^{K-1}}
\geq 2^{-5MNp(M,N)},
\end{align*}
where the last inequality follows from $K\leq M\leq N$ and other coarse bounds.
Therefore, if we pick $\delta:=1-2^{-5MNp(M,N)}$, then since our choice for $\mathcal{K}$ was arbitrary, we conclude that $\Phi$ is $(K,\delta)$-RIP whenever $\mathrm{Spark}(\Psi)>K$, as desired.
Moreover, since $\delta$ can be expressed in the standard representation using $\poly(M,N)$ bits, we can ask the oracle our question in polynomial time.
\end{proof}

It is important to note that Theorem~\ref{thm.rip} is a statement about testing for RIP in the worst case; this result does not rule out the existence of matrices for which RIP is easily verified (e.g., using coherence in conjunction with the Gershgorin circle theorem for small values of $K$~\cite{BandeiraFMW:arxiv12}).


\begin{thebibliography}{00}

\bibitem{CandesT:it05}
E.J. Cand\`{e}s, T. Tao,
Decoding by linear programming,
IEEE Trans. Inform. Theory 51 (2005) 4203--4215.
 
\bibitem{Candes:08}
E.J. Cand\`{e}s,
The restricted isometry property and its implications for compressed sensing,
C. R. Acad. Sci. Paris, Ser. I 346 (2008) 589--592.
 
\bibitem{NeedellT:09}
D. Needell, J.A. Tropp, 
CoSaMP: Iterative signal recovery from incomplete and inaccurate samples, 
Appl. Comput. Harmon. Anal. 26 (2009) 301--321.

\bibitem{BlumensathD:09}
T. Blumensath, M.E. Davies,
Iterative hard thresholding for compressed sensing,
Appl. Comput. Harmon. Anal. 27 (2009) 265--274.

\bibitem{BandeiraFMW:arxiv12}
A.S. Bandeira, M. Fickus, D.G. Mixon, P. Wong,
The road to deterministic matrices with the restricted isometry property,
Available online: arXiv:1202.1234

\bibitem{DeVore:2007}
R. DeVore, 
Deterministic constructions of compressed sensing matrices, 
J. Complexity 23 (2007) 918--925.

\bibitem{Tao:07}
T.~Tao,
Open question: Deterministic UUP matrices, 
Available online: \url{http://terrytao.wordpress.com/2007/07/02/open-question-deterministic-uup-matrices}

\bibitem{KoiranZ:11}
P. Koiran, A. Zouzias,
On the certification of the restricted isometry property,
Available online: arXiv:1103.4984

\bibitem{Cook:online}
S. Cook,
The P versus NP problem,
Available online: \url{http://www.claymath.org/millennium/P_vs_NP/pvsnp.pdf}

\bibitem{McCormick:phd83}
S.T. McCormick,
A Combinatorial Approach to Some Sparse Matrix Problems.
Ph.D. Thesis, Stanford University (1983)

\bibitem{Marx:tcs09}
D. Marx,
A parameterized view on matroid optimization problems,
Theor. Comput. Sci. 410 (2009) 4471--4479.

\bibitem{AlexeevCM:arxiv11}
B. Alexeev, J. Cahill, D.G. Mixon,
Full spark frames,
Available online: arXiv:1110.3548

\bibitem{Khachiyan:jc95}
L. Khachiyan,
On the complexity of approximating extremal determinants in matrices,
J. Complexity 11 (1995) 128--153.

\bibitem{TillmannP:arxiv12}
A.M. Tillmann, M.E. Pfetsch,
The computational complexity of the restricted isometry property, the nullspace property, and related concepts in compressed sensing,
Available online: arXiv:1205.2081


\end{thebibliography}
\end{document}